\documentclass{amsart}[10pt]
\usepackage{epsfig}
\usepackage[english]{babel}
\usepackage{amssymb}
\usepackage{amsmath,amscd}
\usepackage{graphicx}
\usepackage{comment}
\usepackage{amsrefs}
\usepackage{todonotes}
\usepackage[pagewise,left,displaymath,mathlines]{lineno}
\usepackage{xcolor}

\usepackage[pdftex=true, colorlinks=true, pagebackref ]{hyperref}

\newcommand{\R}{\mathbb{R}}

\newcommand{\1}{\mathbb{1}}

\newcommand{\N}{\mathbb{N}}

\newcommand{\C}{\mathbb{C}}
\newcommand{\D}{\mathbb{D}}

\newcommand{\M}{\mathcal{M}}

\newcommand{\K}{\mathcal{K}}

\newcommand{\ds}{\displaystyle}

\newcommand{\logp}{\log_+}

\newcounter{tictac}

\def\1{\,\rlap{\mbox{\small\rm 1}}\kern.15em 1}

\def\build#1_#2^#3{\mathrel{\mathop{\kern 0pt#1}\limits_{#2}^{#3}}}
\def\tend#1#2{\build\hbox to 12mm{\rightarrowfill}_{#1\rightarrow #2}^{ }}

\def\converge#1#2#3#4{\build\hbox to
#1mm{\rightarrowfill}_{#2\rightarrow #3}^{\hbox{\scriptsize #4}}}

\newtheorem{thm}{Theorem}[section]
\newtheorem{prop}[thm]{Proposition}
\newtheorem{lemm}[thm]{Lemma}
\newtheorem{defn}[thm]{Definition}

\newtheorem{rem}[thm]{Remark}
\newtheorem{Cor}[thm]{Corollary}

\newcommand{\beq}{\begin{equation}}
\newcommand{\eeq}{\end{equation}}

\begin{document}
\title[On the Fibonacci Complex Dynamical Systems ]{On the Fibonacci Complex Dynamical Systems}
\author{E. H. El Abdalaoui}
\address{ Department of Mathematics, University
of Rouen, LMRS, UMR 60 85, Avenue de l'Universit\'e, BP.12, 76801
Saint Etienne du Rouvray - France}
\email{elhoucein.elabdalaoui@univ-rouen.fr }

\author{S. Bonnot}
\address{Instituto de Matem\'atica e Estat\'istica,
 Universidade de S\~ao Paulo, Rua do Mat\~ao
1010, S\~ao Paulo, SP, Brasil.}
\email{sylvain@ime.usp.br}

\author {A. Messaoudi}
\address{Departamento de Matim\'atica, IBILCE-UNESP, Rua Cristov\~o Colombo, 2265, CEP 15054-0000, S\~ao
Jos\'e de Rio Preto-SP, Brasil}
\email{messaoud@ibilce.unesp.br}

\author{O. Sester}
\address{LAMA UMR 8050 CNRS, Universit\'e
Paris-Est - Marne-la-Vall\'ee, France}
\email{olivier.sester@univ-mlv.fr}
\date{\today}
\maketitle

\parindent=0cm

{\renewcommand\abstractname{Abstract}
\begin{abstract}
We consider in this paper a sequence of complex analytic functions constructed by the following procedure
$f_n(z)=f_{n-1}(z)f_{n-2}(z)+c$, where $c\in\C$ is a parameter. Our aim is to give a thorough dynamical study of this family,
in particular we are able to extend the familiar notions of Julia sets and Green function and to analyze their properties.
As a consequence, we extend some well-known results.  Finally we study in detail the case where $c$ is small. \\

\vspace{8cm}
\hspace{-0.7cm}{\em AMS Subject Classifications} (2000): 37A15, 37A25, 37A30.
\end{abstract}
\thispagestyle{empty}
\newpage
\section{Introduction}

In this article we 
study a family of complex dynamical systems. This family is constructed using
the so-called multiplicative Fibonacci procedure and we call it the complex Fibonacci dynamical systems.
There are several motivations for studying these mappings.\\
On one hand, they are polynomials endomorphisms of  $\C^2$ defined by $H_c(x,y)=(xy+c,x)$ and constitute an extension of the classical
iteration of polynomials in one complex variable. From the dynamical point of view a natural set to study is $K^+(H_c)=\{(x,y)\in\C^2, H_c^n(x,y) \mbox{ is bounded}\}$ where most
of the non-trivial dynamics concentrates, $H_c^n(x,y)$ is the $n$-th iterate of $H_c$.
We will also show that $H_c$ shares a lot of properties with the
complex H\'enon maps defined by $(x, y)\mapsto (y, P(y)+ c - ax)$ for all $(x,
y) \in \mathbb{C}^2$, where $P$ is a complex polynomial and $a, c$ are fixed
complex numbers (see \cite{BS}, \cite{FS}, \cite{MNTU}).

On the other hand, and this was the starting point of our investigations, complex Fibonacci systems naturally appeared in the context  of stochastic adding machines.
Indeed, Killeen and Taylor in \cite{KT} defined a stochastic adding machine in base $2$ by considering that in the addition of $1$,
the carry  is added with probability $p>0$ and is not added with probability $1-p$.
They obtained a Markov chain and they proved that the spectrum  in $l^{\infty}$ of the transition operator
associated to this Markov chain is equal to the filled Julia set of a quadratic map whose coefficients depend on $p$.

Messaoudi and Smania~\cite{MS} gave a connection of adding machine with
Julia sets in a higher dimensional complex setting. In particular, they
proved that the spectrum $\Sigma$ in $l^{\infty}$  of the transition operator associated to a
stochastic adding machine in an exotic base (given by Fibonacci
numbers) is related to the set $K^{+}(H_c)$ where $H_c(x, y)= (xy+c,
x)$ for all $(x, y) \in \mathbb{C}^2$ whith $c$ a fixed real
number.
More precisely (see \cite{alihoucein}), the spectrum $\Sigma$ contains the set $\K_c=\{z \in \mathbb{C},\; (z,z) \in K^{+}(H_c)\}$, and the conjecture is that $\Sigma= \K_c$.\\

%

In this paper we will lay the basis of a Fatou-Julia theory of those Fibonacci dynamical systems.
To be more specific, let us consider in all what follows
$f:\C\rightarrow\C$ a non constant polynomial function of degree $d$.

\begin{defn}\label{FiboSys} Let us denote $f_0(z)=z$, $f_1(z)=f(z)$ and for any $ n \geq 2$ define $f_n:\C\rightarrow\C$ by the relation
\begin{equation}\label{fn}
f_n(z)=f_{n-1}(z)f_{n-2}(z)+c.
\end{equation}
where $c \in \C$ is a parameter.
\end{defn}
We will call the system given by the above definition, a Complex Fibonacci System (CFS).
%
Here we are mainly interested in the description of topological properties of those Julia sets and of the associated connected locus.

Although the initial problem is a one-dimensional, the behavior of the family $(f_n)$ is closely related to the higher dimensional dynamics of $H_c$, since one has $H^{(n)}_c(z,z)=(f_{n+1}(z),f_n(z))$.
Similarly, the filled-in Julia sets of Fibonacci complex dynamical systems can be seen as a slice of the filled-in Julia sets of $H_c$ :
$\K(f,H_c)= \{z \in \mathbb{C},\; (f(z), z) \in K^{+}(H_c)\}$.

Ergodic properties of the dynamics of $H_c$  were considered by Guedj in \cite{guedj} and \cite{guedjdynamics} . He mentioned there that $H_c$ has a small topological degree (i.e. $\deg(H_c)=1$) and shares this property
with the H\'enon maps (see \cite[Chap.4]{guedj}). Therefore, from the viewpoint of ergodic investigations, the dynamics is not rich and most interesting properties can be established directly, such as the mixing property, hyperbolicity, topological entropy, the existence of the unique measure of maximal entropy, equidistribution of periodic points.
Nevertheless, the study of holomorphic dynamics of this family of maps seems to need delicate arguments.

In this paper, we prove that the set $\K (f,H_c)= \{z \in \mathbb{C},\; (f(z), z) \in K^{+}(H_c)\}$ is a non empty compact set, such that $\C\setminus \K(f,H_c)$ is a connected set.
Moreover, when $c$ is real with  $\vert c \vert >2$ and $f(0)=f'(0)=0$, then $\K(f,H_c)$ is not connected. In the case, where $f(z)=z$, we also show that $\K (f,H_c)$ is not connected, if $c$ is a real number such that $c < -2$.

Our main Theorem asserts  that $\K(f,H_c)$ is a quasi-disk, when $\vert c \vert$ is small and when $f(z)=z$.
In that particular case we will then prove that there exists a real number $\delta>0$ such that if $\vert c \vert < \delta$ then  $\K(f,H_c)$ is a quasi-disk.

Let us mention here that the classical one dimensional machinery used in the study of the usual Julia sets can not be applied in our setting since the graphs $\{f(z),z), z \in \mathbb{C}$ are not invariant under $H_c$.

All our  results can be  extended  easily to the sets $\K(f,h_c)$ associated to the  maps $h_c(x, y)= (x^a y^b+ c, x)$ where $a, b$ are positive real numbers.


In section 2 we shall prove some topological properties of the filled-in Julia set $\K(f,H_c)$. Precisely, we define an "escape radius" and prove that  $\K(f,H_c)$ is a non empty compact and simply connected set.
In section 3 we define the Green function associated to $\K(f,H_c)$ and prove that it has all the expected properties. Section 4 is devoted to the study some
aspect of the connected locus of the Fibonacci dynamical systems. In particular, we prove that if $f(0)=f'(0)=0$ and $\vert c \vert >2$, then $K(f,H_c)$ is not connected.
In the last three sections, we prove our main result that says that if $\vert c \vert $ is sufficiently small and $f(z)=z$, then $\K (f,H_c)$ is a quasi-disk. For $c=0$, $\K(f,H_0)$ simply reduce to the unit disk,
and we are able to define a holomorphic motion from this disk to $\K(f,H_c)$.


\section{On the topological properties of the filled Julia sets of CFS}

Let $f$ be a non null complex polynomial and   $(f_n)_{n \geq 0}$  be the sequence of functions defined by $f_{0}(z)=z, \; f_{1}(z)=f(z)$ and
$f_n(z)=f_{n-1}(z)f_{n-2}(z)+c$, where $c\in\C$ is a fixed complex number.

Set
$$
\K_c = \{z \in\C:~\sup_{n\in \N}| f_n(z)| <+\infty\}.
$$

Observe that
$$
\K_c= \K(f,H_c)= \{z\in\C,\ \mbox{such\ that}\ (f(z),z)\in K^{+}(H_c)\}.$$

where
 $H_c: \mathbb{C}^2 \to \mathbb{C}^2$ is the function defined by $H_c(x, y)= (xy+c,
x)$ for all $(x, y) \in \mathbb{C}^2$ and
$K^{+}(H_c)= \{ z \in \mathbb{C}^2, H_c^{n}(z) \mbox { is bounded }\}$.

The first step is to find an "escape radius"  for the family $(f_n)_{n \geq 0}$: this is the object
of the following  crucial proposition.
\begin{prop}[Escape Radius]{\label{FilledJulia}}
There exists $R = R(c)>1$,  such that

\[
\K_c=\bigcap_{n=0}^{+\infty}f_n^{-1}\left(\overline{\D(0,R)}\right).
\]

\end{prop}

\begin{lemm}\label{twice} Let $M > max \{2, \sqrt{2 \vert c \vert}\}$. If the set $\ds \left\{n \in \mathbb{N} :\ |f_n(z)| >M \right\}$ contains two consecutive integers then the sequence $(f_n(z))_{n \geq 0}$ is unbounded.
\end{lemm}
\begin{proof}Let $k,~ k+1$ be  two consecutive integers in $\ds \left\{n \in \mathbb{N}~:~ |f_n(z)| >M \right\}$. Then, by the triangle inequality, we have
\begin{eqnarray*}
|f_{k+2}(z)| \geq  M^2-|c| > \frac{M^2}2.
\end{eqnarray*}
Hence
by the same reasoning we get
\begin{eqnarray*}
|f_{k+3}(z)| &\geq&  \frac{M^3}2-|c| \\
             & >  &  \left(\frac{M^2}2+|c|\right)\frac{M}2-|c|>2\left(\frac{M}2\right)^3.
\end{eqnarray*}
Applying the same procedure one may easily deduce by induction that
\[
 \forall n \in \mathbb{N},\; |f_{k+n}(z)| > 2\cdot {\left(\frac{M}2\right)}^{{F}_n}
\]
where ${F}_n$ is the $n$-th Fibonacci number with initial values $F_0= F_1=1$. The proof of the Lemma is thus completed.
\end{proof}

\begin{lemm}\label{technique} Let $P$ be a non constant  polynomial function. There exist positive real numbers $\delta= \delta(P),\; p= p(P) >1$, ~$  d= d(P)>0$ such that
for all $E \geq \delta$ and for all $z \in \mathbb{C}$ we have
\begin{eqnarray*}
\vert z \vert > E  \Rightarrow \vert P(z)\vert >d E \mbox { and } \vert  P(z)  \vert > E^p \Rightarrow \vert z \vert >E.
\end{eqnarray*}
\end{lemm}

\begin{proof}
The first inequality comes from the fact that the function $\frac{\vert P(z) \vert}{\vert z \vert}$ converges either to a positive constant or to infinity as $\vert z \vert $ goes to infinity.
The second inequality is due to the fact that
if $E >1$ is a real number and   $\vert z \vert \leq E$ then $ \vert  P(z)  \vert \leq c_1 E^k $ where $k$ is the degree of the polynomial $P$ and  $c_1$ is the sum of the modulus of the coefficients of $P$.
\end {proof}

\begin{prop}{\label{Keyprop}} Let $M$ be as in Lemma~\ref{twice} then
there exists $R>M$ such that if $|f_{k}(z)| > R$, for some integer $k$, then the sequence $(f_n(z))_{n \geq 0}$ is not bounded.
 \end{prop}

\begin{proof} Consider $\delta, p, $ and $d$ given by Lemma~\ref{technique} applied to $P=f_1$.
Let $M> \max\{ 2, \sqrt{2 \vert c \vert},  \delta\}$ and $R > \max \{M+ 2 \vert c \vert, \frac{M}{d}, M^p\}$.
If $| f_k(z) | >R$ for $k=0$ or $1$, then by  Lemmas~\ref{twice} and~\ref{technique}, we have $(f_n(z))_{n \geq 0}$ is unbounded.

Now, assume that the property is true for $1 \leq n <k$ and that $|f_{k}(z)| >R$.
It follows from Lemma~\ref{twice} that
$$
\max \{|f_{k-1}(z)|,|f_{k+1}(z)|\} \leq M .
$$
\noindent Applying the triangle inequality we obtain
\begin{eqnarray}\label{Fiboeq1}
|f_{k-1}(z)| < \frac{|f_{k+1}(z)|+|c|}{R} \leq \frac{M+|c|}{R}.
\end{eqnarray}
\noindent On the other hand, we have
\[
R <|f_k(z)| <  \left(\frac{M+|c|}{R}\right)|f_{k-2}(z)|+|c|.
\]
\noindent It follows that
$
|f_{k-2}(z)|>  R. \frac{R-|c|}{M+|c|}> R.$
Hence $(f_n(z))_{n \geq 0}$ is unbounded.

\end{proof}

The proof of Proposition~\ref{FilledJulia} is now an immediate consequence of Proposition~\ref{Keyprop}.\\
\begin{Cor}
$\K_c$ is a compact subset of $\C$.
\end{Cor}
\begin{proof} This corollary is straightforward from Proposition~\ref{FilledJulia}.

\end{proof}

\begin{prop}
The set $\mathbb{C}\setminus \K_c$ is a connected set.
\end{prop}

\begin{proof}
By Proposition \ref{FilledJulia}, there exists a real number $R >0$ such that
$\mathbb{C} \setminus \K_c= \bigcup_{n=0}^{+\infty} \mathbb{C} \setminus
f_{n}^{-1} (\overline{D(0, R)}).$ Now, for all $n\geq 0,
\; \mathbb{C} \setminus f_{n}^{-1} (\overline{D(0, R)})$ contains a neighborhood of infinity and is a connected set: indeed its complement $\{z; \vert f_n(z) \vert \leq R \}$ is necessarily simply connected (by the maximum principle applied to $f_n$).

\end{proof}

\begin{prop}
\label{encaixante}
There exists a real number $R >0$ such that
 $$\K_c=
\bigcap_{n=1}^{+\infty} f_{n}^{-1} (D(0, R)) \mbox {  and
} f_{n+1}^{-1} (D(0, R)) \subset f_{n}^{-1} (D(0,
 R)), \; \forall n \geq 1.$$
\end{prop}

\begin{lemm}
\label{rrrt}
There exists a real number $R_0 >0$ such that for $R >R_0$,
and for all $k \in \{1,2,3,4\}$, we have
\begin{eqnarray}
\label{dd}
\vert f_{k+1}(z) \vert <R \Longrightarrow \vert f_{k}(z) \vert <R .
\end{eqnarray}
\end{lemm}

\begin{proof}

Consider $\delta(f_i) $, $ p(f_i)$ and $d(f_i)$ the real numbers defined in Lemma~\ref{technique} corresponding to $f_i$ for $i=1,..,4$. Let $R_1 > \delta (f_1)$  such that for all $R >R_1$, we have $R^{\frac{1}{p(f_1)}+1}- \vert c \vert \geq R$.
Fix  $R >R_1$, and $z \in \mathbb{C}$ such that $\vert f_1(z) \vert \geq R$, then by Lemma~\ref{technique},
$ \vert f_0(z) \vert =\vert z \vert \geq R^{1/p (f_1)} $ ,
hence $\vert f_{2}(z) \vert \geq  R^{\frac{1}{p (f_1)}+1 }- \vert c \vert \geq R$, then the result is true for $k=1$.

Let $R_2 > \delta (f_2)$  such that for all $R >R_2$, we have $d(f_2) R^{\frac{1}{p(f_2)}+1}- \vert c \vert \geq R$.

Consider $R >R_2$  and $z \in \mathbb{C}$ such that $\vert f_2(z) \vert \geq R$, then by Lemma \ref{technique} we obtain that $\vert z \vert \geq  R^{1/p(f_2)}$
hence again by Lemma \ref{technique}, we have  $\vert f_{1}(z)  \vert \geq d(f_2) R^{1/p(f_2)}.$
Thus
$\vert f_3(z) \vert \geq d(f_2) \vert  R^{\frac{1}{p(f_2)}+1}- \vert c \vert  \geq R$, then we have the result for $k=2$.
By the same way, if we choose for $i=3,4,\; R_i > \delta (f_i)$  such that for all $R >R_i$, we have $d(f_i) R^{\frac{1}{p(f_i)}+1}- \vert c \vert \geq R$.
 We prove that the result holds for $k=3$ and $4$.
 Taking $R_0= max\{R_1, R_2, R_3, R_4\}$, we conclude the proof of Lemma \ref{rrrt}.

\end{proof}

{\bf Proof of Proposition \ref{encaixante}}. Let $R >R_0$ be a large number such that
$\K_c=\bigcap_{n=0}^{+\infty} f_{n}^{-1} (D(0, R)).$
 Let us prove by induction that
  for all $k \in \mathbb{N},\;
f_{k+1}^{-1} (D(0, R)) \subset f _{k}^{-1}( D(0,R))$.

By Lemma \ref{rrrt}, the previous inclusion is true for $k=0,1,2,3$. Now
  assume that it is also true for all $k=0, \ldots,
n-1,\; n \geq 4.$

 Let $ z \in f_{n+1}^{-1} (D(0,R))$, then $\vert f_{n+1}(z) \vert < R.$ Assume
that $\vert f_{n}(z) \vert \geq R.$ By using relation (\ref{fn}) and
a triangle inequality, we have
\begin{eqnarray}
 \label{xn}
\vert f_{n-1}(z) \vert <\frac{R+ \vert c \vert}{R}=1+ \frac{ |c|}{R}=O(1).
 \end{eqnarray}

We choose $R$ sufficiently large such that $\frac{R+ \vert c \vert}{R}  \leq R,$
then $\vert f_{n-1}(z) \vert < R.$ Hence, by induction hypothesis
\begin{eqnarray}
\label{relk} \vert f_{k}(z) \vert < R,\; \forall k=0,\ldots, n-2.
\end{eqnarray}
On the other hand,
by using relations (\ref{fn}) and (\ref{xn}), we have

 \begin{eqnarray}
\label{xn2} \vert f_{n-2}(z) \vert > R \frac {R - \vert c \vert }{R+
\vert c \vert}:= h_1(R).
\end{eqnarray}

 By  relations (\ref{fn}), (\ref{xn}) and (\ref{xn2}), we deduce
that
 \begin{eqnarray}
\label{xn3}
\vert f_{n-3}(z) \vert < \frac {(R (1+ \vert c \vert) + \vert c \vert) (R+\vert c \vert )}{R^2 (R- \vert c \vert)}:= h_2(R).
\end{eqnarray}

 By  using relations (\ref{fn}), (\ref{xn2}) and (\ref{xn3}), we deduce
\begin{eqnarray}
 \label{xn4}
\vert f_{n-4}(z) \vert \geq \frac{h_1 (R)- \vert c \vert }{h_2( R)}=  O(R^2) >R.
\end{eqnarray}
 This contradicts relation (\ref{relk}), and we obtain the result.

\hfill $\Box$

\begin{rem}
\begin{enumerate}

\item
It is easy to see that the Julia set $J_c:=\partial \K_c$ is exactly the set of points $z\in\C$ such that $(f_{n})_{n\in\N}$
does not form a normal family in a neighborhood $z$.
\item If $f$ is a constant non null polynomial, then the previous results are still true, in particular $\K_c$ is a non empty compact and simply connected set.

If $f=0$, then $\K_c= \mathbb{C}$ if $0\in \K_c$ and  $\K_c= \emptyset$ if $0 \not \in \K_c$.

\end{enumerate}
\end{rem}

\begin{figure}[htb]
\begin{center}
\includegraphics[width=5cm]{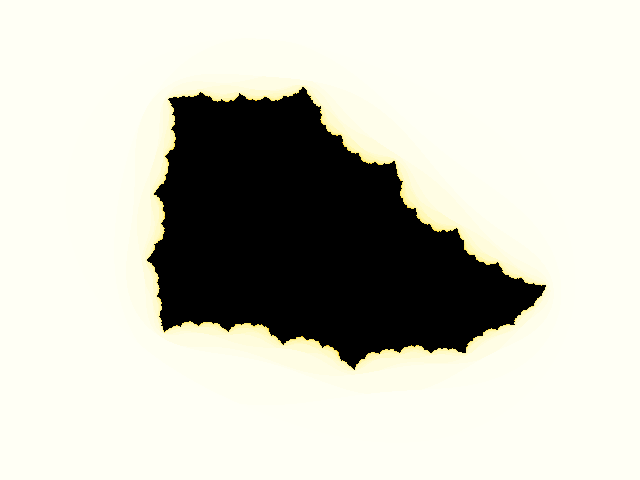}
\includegraphics[width=5cm]{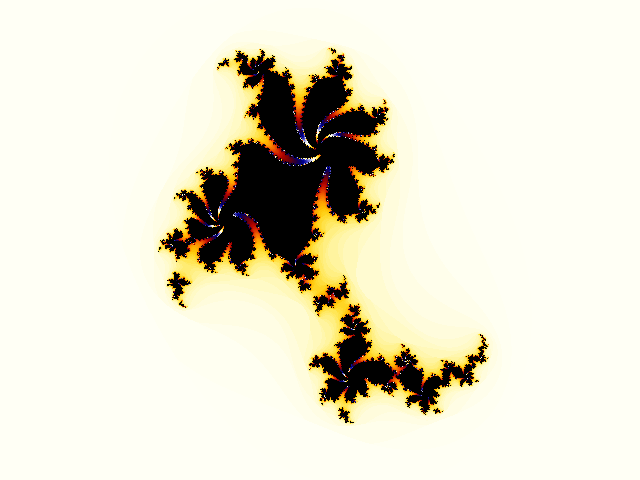}
\caption{Filled-in Julia set intersected with the diagonal, for $c=-0.5+i0.5$ in the left and $c=0.36+i0.575$ in the right}
\label{Filled-in-Julia-set}
\end{center}
\end{figure}

\section{Green's Function}
In this section we introduce the Green's function of the filled-in Julia set, which is the main tool to relate potential theory and dynamics.\\
According to the theory of polynomial endomorphisms in $\C^2$, the Green's function of $H_c(w,z)=(wz+c,w)$ is well defined  and has nice
properties.
But here we will provide a new characterization and be more specific on the Green's function associated to our Fibonacci System.\\
Let $\log_+$ be the real function defined by $\log_+(x)=\max(\log|x|,0)$. For all $(w,z)\in\C^2$, we
denote $H_{c}^n(w,z): =(h_{n+1}(w,z),h_{n}(w,z))$ for all $n\geq 0$.\\
With this notation $f_{n}(z)=h_n(f(z),z)$. Let us denote $d_{n}$ the degree of $f_n$.  These numbers satisfy $d_0=1$, $d_1=\deg(f)$ and the well known Fibonacci relation $d_{n+1}=d_{n}+d_{n-1}$ for all $n\geq 1$.\\

Before introducing the Green's function we need to establish some control on $h_n$ and $f_n$. The first tools that we  will use are the following technical Lemmas.
\begin{lemm}\label{controlhn}
There exists a real number $R >1$ such that for all $|w|, |z|>R$ and for all $n\geq 0$,
\begin{equation}\label{majorationhn}
R\leq |h_{n}(w,z)|\leq \frac{e^{d_n}}{2}|z|^{d_n}|w|^{d_n}.
\end{equation}
\begin{proof}
Let $R >1$ such that $R^2 - \vert c \vert > R$, then the lower bound follows by the same
reasoning as in Lemma~\ref{twice}.
The upper bound is proved by induction.
Since $h_{0}(w,z)=z$ and $h_{1}(w,z)=w$ the claim is straightforward for $n=0$ and 1. Now assume that~(\ref{majorationhn}) holds
for the steps $n-1$ and $n$,
$$
|h_{n+1}(z)|\leq |h_n(z) h_{n-1}(z)|+|c| \leq\frac{e^{d_{n+1}}}{4}|z|^{d_{n+1}}|w|^{d_{n+1}}+|c|\leq \frac{e^{d_{n+1}}}{2}|z|^{d_{n+1}}|w|^{d_{n+1}}.
$$
For this last estimate we used $2|c|\leq R^2\leq |z| |w|.$
\end{proof}

\end{lemm}
\begin{lemm}\label{majoration}
There exist  positive real numbers $A$ and $R$  such that for all $n>0$ and for all $z\in\C$
\begin{equation}\label{majorationf}
\left|\logp |f_{n+1}(z)|-\log_+|f_{n}(z)|-\log_+|f_{n-1}(z)|\right|\leq A,
\end{equation}
and similarly,  for all $|z| \geq R$ and $|w|\geq R$ :
\begin{equation}\label{majorationh}
\left|\logp |h_{n+1}(w,z)|-\log_+|h_{n}(w,z)|-\log_+|h_{n-1}(w,z)|\right|\leq A.
\end{equation}
\end{lemm}
\begin{proof}
Obviously (\ref{majorationf}) and (\ref{majorationh}) are proved in a very similar way. First we focus on (\ref{majorationf}).
Fix $z \in \mathbb{C}$ and
consider $R >1$ as in Proposition~(\ref{encaixante}).
If $\vert f_{n-1} (z) \vert \geq R$ for some integer $n \geq 1$, then $\vert f_{n} (z) \vert \geq R$.  There are only 3 cases to consider. \\
Either, $|f_{n-1}(z)| \geq R$, then in this case  $|f_{n}(z)| \geq R$. Hence
$$
 \frac{|f_{n+1}(z)|}{|f_n(z)| |f_{n-1}(z)|} \leq \left| 1+  \frac{|c|}{|f_n(z)| |f_{n-1}(z)|} \right|\leq1+\frac{|c|}{R^2}.
$$
Hence
\begin{eqnarray*}
\left|\logp |f_{n+1}(z)|-\logp|f_{n}(z)|-\logp|f_{n-1}(z)|\right|& = & \left|\log( \frac{|f_{n+1}(z)|}{|f_n(z)| |f_{n-1}(z)|} )\right| \\
  									& \leq & \log(1+\frac{|c|}{R^2} ):=A_1.
\end{eqnarray*}
The second case is  $|f_{n-1}(z)| <R$ and $|f_{n}(z)| \leq R$, then
$
|f_{n+1}(z)| < R^2+|c|\leq R^3.$\\
Therefore
$$
\left|\logp |f_{n+1}(z)|-\logp|f_{n}(z)|-\logp|f_{n-1}(z)|\right|\leq 2 \log R+ \log(R^3)=5 \log R:=A_2.
$$
The last case to consider is $|f_{n-1}(z)|< R$ and $|f_{n}(z)| > R$, then $|f_{n+1}(z)| \geq R$,
then
$$
R \leq |f_n(z)f_{n-1}(z)|+|c|
$$
thus we obtain $|f_n(z)f_{n-1}(z)|\geq R-|c|$, then
$$
 \frac{|f_{n+1}(z)|}{|f_n(z)| |f_{n-1}(z)|}\leq 1+\frac{|c|}{R-|c|}.
$$
Taking the $\logp$ we conclude.\\
Regarding (\ref{majorationh}), we assume now that both $|z|$ and $|w|$ are greater than $R$. According
to Lemma~\ref{controlhn}
$
|h_{k}(w,z)|\geq R.
$
Thus we are precisely in the first
case above and we can proceed in the same way.
\end{proof}


\begin{prop}\label{normale}
There exists a constant $C>0$ such that, for all $n\geq 0$ and $z \in \mathbb{C}$
$$
\left|\frac{1}{d_{n+1}}\logp|f_{n+1}(z)|-\frac{1}{d_{n}}\logp|f_{n}(z)|\right| \leq\frac{C}{d_{n+1}}.
$$
\end{prop}

\begin{proof}

This key proposition is based on the following identity.
\begin{eqnarray*}
\logp|f_{n+1}(z)|-\frac{d_{n+1}}{d_{n}}\logp|f_{n}(z)| & = & \logp |f_{n+1}(z)|- \logp|f_{n}(z)|-\logp|f_{n-1}(z)| \\
											& & +\sum_{k=0}^{n-2}(-1)^{k+1}\frac{d_{n-k-1}}{d_{n}}(\logp |f_{n-k}(z)|\\
& & -\logp|f_{n-k-1}(z)|-\logp|f_{n-k-2}(z)|) \\
											& & +\frac{(-1)^n}{d_{n}}(d_0\logp |f_{1}(z)|-d_1\logp |f_{0}(z)|).\\	
\end{eqnarray*}
It's a simple calculation to check this equality. Indeed, if we re-arrange the terms of $\logp |f_{n-k}|$ for $1<k<n-1$, we obtain
$$
\frac{(-1)^{k}} {d_{n}}( - d_{n-k-1}-d_{n-k} + d_{n-k+1})=0.
$$
But $(d_{n})_{n \geq 0}$ is a Fibonacci sequence, so we obtain :
$$
\left| \frac{d_{n-k-1}}{d_{n}}\right| \leq\frac{D}{\rho^k} \mbox{ where } \rho=\frac{1+\sqrt{5}}{2}  \mbox{ is the golden mean and } D >0   \mbox{ is a real constant}.
$$
From this estimation combined with Lemma~\ref{majoration} and the identity above we deduce :
\begin{equation}\label{interm}
\left| \logp|f_{n+1}(z)|-\frac{d_{n+1}}{d_{n}}\logp|f_{n}(z)| \right|\leq \sum_{k=0}^{n} \frac{D A}{\rho^k}+| \frac{1}{d_{n}}(d_0\logp|f_{1}(z)|-d_1\logp|f_{0}(z)|) |.
\end{equation}
Let $I(z):=d_0\logp|f_{1}(z)|-d_1\logp|f_{0}(z)|$.
For $|z|\leq R$, $|I(z)|$ is clearly bounded from above. On the other hand, whenever $|z|>R$ then $I(z)=\log|f_{1}(z)|-d_1\log|z|$,
Let
$$f_1(z)= a_{d_{1}}z^{d_1}+\ldots + a_1 z+a_0=z^{d_1}(a_{d_1} +\ldots +a_1 z^{1-d_1}+a_0z^{-d_1}), 
$$
then
$$
I(z)=\log(\left| a_{d_{1}} +a_{d_1-1}z^{-1}+\ldots a_1 z^{1-d_1}+a_0z^{-d_1}\right|).
$$
Thus, if $|z|>R$ we obtain $|I(z)|\leq \log(\max(\vert a_i \vert) (1+ d_1/ R))$.

Using (\ref{interm}), this proves immediately that  $\left| \logp|f_{n+1}(z)|-\frac{d_{n+1}}{d_{n}}\logp|f_{n}(z)| \right|$ is bounded above by a constant.
 We obtain therefore the expected result dividing this inequality by $d_{n+1}$.
\end{proof}

In the same spirit as the previous proposition one can show, the following:
\begin{prop}
\label{normale2}
There exist  real numbers $C>0$ and $R >1$ (as in Proposition \ref{encaixante}) such that  for all $n\geq 0$ and all $|w|, |z|>R$,
\begin{equation}\label{normale4}
\left|\logp|h_{n+1}(w,z)|-\frac{d_{n+1}}{d_{n}}\logp|h_{n}(w,z)|\right| \leq C
\end{equation}
and
\begin{equation}\label{normale3}
\left|\logp|h_{n+1}(w,z)|-\rho \logp|h_{n}(w,z)|\right| \leq C_1+C_2\rho^{-n} (\log|z|+\log|w|), \textrm{ where } C_1, C_2 \in \mathbb{R}.
\end{equation}
\end{prop}

\begin{proof}
The first inequality is proved  in the same way as Proposition~\ref{normale} and we leave it to the reader. Now we deduce (\ref{normale3}) from inequality~(\ref{normale4}).
Recall $d_n=\lambda_0\rho^n+\lambda_1  (-\rho)^{-n}$ where $\lambda_0,\lambda_1$ are real numbers,
Thus
$$
\frac{d_{n+1}}{d_n}=\frac{\lambda_0\rho^{n+1}+\lambda_1(-\rho)^{-n-1}}{\lambda_0\rho^n+\lambda_1(-\rho)^{-n}}= \rho + O(\rho^{-2n}).$$
Replacing this expression of $\frac{d_{n+1}}{d_n}$ in the estimate~(\ref{normale4}), we deduce
$$
\left|\logp|h_{n+1}(w,z)|-\rho \logp|h_{n}(w,z)|\right| \leq {C} + O(\rho^{-2n})\logp|h_{n}(w,z)|.
$$
From Lemma~\ref{controlhn} one deduce for all $|z|, |w|\geq R$,
$$
\logp|h_{n}(w,z)|\leq d_n+d_n(\log|z|+\log|w|).
$$
This yields
 $$
\left|\logp|h_{n+1}(w,z)|-\rho \logp|h_{n}(w,z)|\right| \leq {C}+ O(\rho^{-n})(\log|z|+\log|w|).
$$
\end{proof}
Now, let us define the (dynamical) Green's function of the compact $\K_c$, for all $z\in\C$:
$$
g(z)=g_c(z):=\lim_{n\rightarrow+\infty}  \frac{1}{d_{n}}\logp|f_n(z)|.
$$
This definition makes sense. Indeed,
\begin{thm} \label{green}
The function $g_c:\C\rightarrow\R_+$ satisfies the following properties:
\begin{enumerate}
\item $g_c$ is harmonic in $\C\setminus \K_c$;
\item $\K_c$ is exactly the set $g_c^{-1}(\{0\})$;
\item $g_c(z)-\log|z|$ tends to the constant $(\log \vert a_{d_1} \vert) \sum_{n=0}^{+\infty} \frac{(-1)^n}{d_{n}d_{n+1}}$ as $|z|$ tends to $+\infty$, where $f(z)= f_1(z)= a_{d_1} z^{d_1}+ \cdots+ a_0$;
\item $(c,z)\mapsto g_c(z)$ is continuous.
\end{enumerate}
\end{thm}
\begin{proof}
We forget obvious dependencies on $c$ when not necessary. The first point is more or less obvious, it suffices to write $g$ as the series
\begin{equation}\label{serie-green}
g(z)-\logp|z|=\sum_{n=0}^{+\infty}g_{n+1}(z)-g_{n}(z)
\end{equation} and in view of Proposition~(\ref{normale}) the functions $g_n$:
$$
g_n(z)=\frac{1}{d_{n}}\logp|f_n(z)|,
$$
form a sequence of harmonic functions that converge uniformly to $g$.\\
Clearly, if $(f_{n}(z))_{n\in\N}$ is bounded then $g(z)=0$.
The converse requires to be a little more cautious. Assume $z\notin\K_c$, according to Proposition~\ref{encaixante} there exists
$n_0$ such that both $|f_{n_0}(z)|$ and $|f_{n_0+1}(z)|>R$. Let $\lambda=\frac{\ds \min\{|f_{n_0}(z)|,|f_{n_0+1}(z)|\}}{R}$. Obviously $\lambda>1$. \\
We claim that for all $k\geq 0$, $|f_{n_0+k}(z)|>\lambda^{{F}_k} R$ where ${F}_k$ is the Fibonacci number with initial terms $F_0=F_1=1$. For $k=0$ and $1$ this
is just the definition of $\lambda$. Now by induction, assume  for all $p\leq k,\; |f_{n_0+p}(z)|>\lambda^{{F}_p} R$. Then
\begin{eqnarray*}
|f_{n_0+k+1}(z)| & \geq &|f_{n_0+k}(z) f_{n_0+k-1}(z)|-|c|\\
			&>& \lambda^{{F}_{k+1}}R^2-|c|
			\geq   \lambda^{{F}_{k+1}}R.
\end{eqnarray*}
For this last estimate we use the fact that $R^2-|c|\geq R$, and the claim follows.

Now we return to the Green's function:
$$
g_{n_0+k}(z)\geq \frac{{F}_k}{d_{n_0+k}}\log\lambda
$$
and this guarantees that $g(z)=\lim_{k\rightarrow\infty} g_{n_0+k}(z)>0$ whenever $z\notin\K_c$.
\\
Concerning the third point, each term in~(\ref{serie-green}) goes to zero as $n\rightarrow +\infty$. Indeed, let us write
$f_n(z)= a_{d_{n},n} z^{d_{n}}+a_{{d_{n}-1}, n} z^{d_{n}-1}+\ldots+a_{0,n}$ for all $n \geq 0.$
Since $f_{0}(z)=z$, it is easy to see by induction that $a_{d_{n},n}= a_{d_{1},1}^{F_{n-1}}$ for all integer $n \geq 1$.
Therefore

$
g_{n+1}(z)-g_{n}(z)=
(\frac {F_{n}}{d_{n+1}}- \frac {F_{n-1}}{d_{n}}) \log \vert a_{d_{1},1}\vert + \frac {1}{d_{n+1}} \logp \vert f_{n+1}(z) / a_{d_{n+1},n+1} z^{d_{n+1}}\vert -
\frac{1}{d_{n}} \logp \vert f_{n}(z) / a_{d_{n},n} z^{d_{n}}\vert.
$
On the other hand, for all integer $n \geq 2$,
$$
\left (
\begin{array}{cc}
F_{n-1}& F_n \\
 d_n& d_{n+1}
\end{array}
\right)
=
\left (
\begin{array}{cc}
F_{n-2}& F_{n-1} \\
 d_{n-1}& d_{n}
\end{array}
\right)
\left (
\begin{array}{cc}
0& 1\\
 1& 1
  \end{array}
\right)
$$
Hence
$$
\left (
\begin{array}{cc}
F_{n-1}& F_n \\
 d_n& d_{n+1}
\end{array}
\right)
=
\left (
\begin{array}{cc}
F_{0}& F_{1} \\
 d_{1}& d_{2}
\end{array}
\right)
M^{n-1}$$
where
$M=
\left (
\begin{array}{cc}
0& 1\\
 1& 1
  \end{array}
\right)
$.
Thus $F_n d_n- F_{n-1}d_{n+1}= (F_1 d_1- F_0 d_2) (-1)^{n-1}= (-1)^n$.
Therefore, as $\vert z \vert$ converges to $+\infty$,  $g_{n+1}(z)-g_{n}(z)$ converges  to
$\frac {(-1)^n}{d_{n}d_{n+1}}  \log \vert a_{d_{1},1}\vert$.
Since the series is normally convergent we
deduce that $g(z)-\log(|z|)$ tends to  $\log \vert a_{d_{1},1}\vert \sum_{n=0}^{+\infty} \frac{(-1)^n}{d_{n}d_{n+1}}$ as $\vert z \vert \rightarrow +\infty$.
Each $g_n$ depends continuously  on $(z,c)$, thus the sum of
the series is also continuous.

\end{proof}

\begin{rem}
According to Proposition~\ref{normale2},  we can also define the 2-dimensional Green's function associated to $\phi$. For all $|z|, |w|>R$, let
$$
G(w,z):=\lim_{n\rightarrow+\infty}  \frac{1}{d_{n}}\logp|h_n(w,z)|.
$$
This formula makes sense and $G$ and $g$ are related by $G(f(z),z)=g(z)$. Moreover, $G$ satisfies
the functional equation $G\circ H_c(w,z)=\rho G(w,z)$. This fact can be interpreted in connection with the existence of invariant measure on the Julia set for which the action of $H_c$ on the Julia set is mixing and with maximal entropy given here by $\log(\rho)$, such probability measure exists by virtue of Dinh-Sibony Theorem (see \cite{guedj} and the references therein).
 One can obtain an alternative proof of Proposition 2.6 using the properties of $G$. We shall include the proof here for the convenience of the reader (see Proposition 3.8).
\end{rem}
\begin{rem}
Recall that we denote $\rho=\frac{1+\sqrt5}{2}$,  $d_n$ is equal to $\lambda_0\rho^n+\lambda_1((-1/\rho)^n)$ with
$\lambda_0,  \lambda_1$ constants which essentially depends on $d_1$. Thus $\frac{d_n}{\rho^n}$ tends to $\lambda_0$ and we could replace
$d_{n}$ by $\lambda_0 {\rho^n}$ in the definition of the Green's function
$$
g(z)=\lim_{n\rightarrow+\infty}  \frac{1}{\lambda_0 \rho^n}\logp|f_n(z)|.
$$
\end{rem}
Theorem~\ref{green} admits several consequences.
\begin{prop}
The filled-in Julia set is a full compact subset of $\C$ i.e. the complementary has no bounded component.
 The logarithmic capacity of the filled-in Julia set is equal to $e^{\sigma}$ where
$\sigma=\log |a_{d_1}| \sum_{n=0}^{+\infty} \frac{(-1)^n}{d_{n}d_{n+1}}.$
\end{prop}
\begin{proof} The proof is word to word the same as for polynomial dynamical systems. Indeed, assume there exists $\mathcal O$ a bounded connected component of $\C\setminus\K_c$. Then by the maximum principle
$$
\max_{z\in \mathcal O} g(z)=\max_{z\in \partial\mathcal O} g(z)=0
$$
as $\partial\mathcal O\subset \K_c$. This provides a contradiction since $\mathcal O$ is not included in $\K_c$.\\
The logarithmic capacity is defined as $e^\sigma$ where $\sigma$ is the constant in
$$
g(z)=\log(|z|)+\sigma+o(1),
$$
and according to Theorem~\ref{green} we obtain $\sigma=\log |a_{d_1}| \sum_{n=0}^{+\infty} \frac{(-1)^n}{d_{n}d_{n+1}}.$

\end{proof}
An other consequence of Theorem~\ref{green} is the following. Let $Comp^*(\C)$ denotes
the set of non empty compact Hausdorff subsets of
$\C$.
\begin{prop}
The function $c\mapsto J_c$ from $\C$ to $Comp^*(\C)$ is lower semi-continuous.
\end{prop}
\begin{proof}
The standard proof works in our setting. Indeed, the map $c\mapsto g_c$ is continuous in $L^1$ according to Theorem~\ref{green}, and $c\mapsto \Delta g_c$ is
also continuous in the sense of distributions.  Hence the support of the measure $\Delta g_c$ is lower semi-continuous
with respect to $c$. Moreover, $g_c$ is harmonic in $\C\setminus \K_c$ and $g_c=0$ on the interior of $\K_c$, thus the support of
$\Delta g_c$ coincides with $J_c$.
\end{proof}

We can also define the analogue of the so-called B\"ottcher coordinates in a neighborhood of infinity by the formula
\begin{equation}\label{bottcher}
\varphi(w,z)=z\prod_{k=0}^{+\infty} \left(\frac{h_{k+1}(w,z)^{\frac{1}{\rho^{k+1}}}}{h_k(w,z)^{\frac{1}{\rho^{k}}}}\right).
\end{equation}
\begin{prop}\label{Sibony}
There exists a real number $R>1$ such that the function $\varphi$ is well defined for all $|w|, |z|\geq R$, and is an analytic function of $(w,z)$ that satisfies the functional equation
$$
\varphi\circ H_c(w,z)=\varphi^{\rho}(w,z).
$$
Moreover, $G(w,z)=\lambda_0\log|\varphi(w,z)|$.
\end{prop}
\begin{proof}
Assume $\varphi$ is given by (\ref{bottcher}) and that the product converges. We compute
$$
\varphi\circ H_c(w,z)=w\prod_{k=0}^{+\infty} \left(\frac{h_{k+1}(\phi(w,z))^{\frac{1}{\rho^{k+1}}}}{h_k(\phi(w,z))^{\frac{1}{\rho^{k}}}}\right)
=z^\rho \frac{w}{z^\rho}\prod_{k=0}^{+\infty} \left(\frac{h_{k+2}(w,z))^{\frac{1}{\rho^{k+1}}}}{h_{k+1}((w,z))^{\frac{1}{\rho^{k}}}}\right).
$$
Recall $h_0(w,z)=z$ and $h_1(w,z)=w$, thus we deduce the functional equation
$$
\varphi\circ H_c(w,z)= z^\rho\prod_{k=0}^{+\infty} \left(\frac{h_{k+1}(w,z)^{\frac{1}{\rho^{k+1}}}}{h_k(w,z)^{\frac{1}{\rho^{k}}}}\right)^\rho=\varphi^{\rho}(w,z).
$$
Concerning the convergence of the product (\ref{bottcher}), it should be noticed that for all $|z|, |w|\geq R$, $|h_k(w,z)|\geq R$. On the other hand,
\begin{eqnarray*}
\left|\frac{h_{k+1}(w,z)^{\frac{1}{\rho^{k+1}}}}{h_k(w,z)^{\frac{1}{\rho^{k}}}}\right| &= &\exp \left(\frac{1}{\rho^{k+1}}
\left|\logp|h_{k+1}(w,z)|-\rho \logp|h_{k}(w,z)|\right| \right)\\
&\leq & \exp\left(\frac{C(\log|z|+\log|w|+1)}{\rho^{k+1}}\right)
\end{eqnarray*}
where $C \in \mathbb{R}$ and this last estimate results from~(\ref{normale3}). Thus we deduce that the product (\ref{bottcher}) converges to an analytic function.
\end{proof}

Of course, for our purpose Proposition \ref{Sibony} is mostly interesting  when applied to $w=f(z)$.
\section{Properties of the connectedness locus of CFS}

\begin{prop}
\label{discon1}
Assume that $f_1(z)= f(z)$ satisfies $f(0)=f'(0)=0$, then for all complex numbers $c$ such that $|c| >2$,
we have that $\K_c$ is a disconnected set.
\end{prop}

\begin{proof}

Since $f(0)=f'(0)=0$, it is easy to check that $f_{n}'(0)= 0$ for all $n \geq 1$.
Using Riemann-Hurwitz Formula and Proposition \ref{encaixante}, we deduce that if $0 \not \in \K_c$, then $\K_c$ is not connected.

Assume, now that $\vert c \vert >2$, then we have by induction the following claim:

{\bf Claim:} $\vert f_n(0)\vert \geq (\vert c \vert-1)  \vert f_{n-1}(0)\vert,\; \forall n \geq 4$.

Indeed,
$ \vert f_{4}(0)\vert= \vert c^2+c\vert \geq (\vert c \vert-1) \vert f_{3} (0) \vert.$
Suppose that the claim is true for all integer  $4 \leq k \leq n$,
since $\vert c \vert >2$ and $f_3(0)= c$, we deduce that $\vert f_k (0)\vert \geq \vert c \vert$ for all $3\leq k\leq n$.
Hence
$$\vert f_{n+1}(0)\vert \geq  \vert f_{n}(0) f_{n-1}(0)\vert - \vert c \vert \geq (\vert c \vert-1)  \vert f_{n-1}(0)\vert.$$
Then, we obtain the claim.

\end{proof}

\vspace{ 0.5 em}

{\bf Question.}
\begin{enumerate}
\item The condition  that $0$ is a zero of multiplicity two is crucial in our proof. One may ask if this condition
 can be relaxed.

 \item In our proof we have proved that $K_c$ is disconnected for $|c|>2$. One may ask what is the minimal
 constant $d$ such that  for all $\vert c \vert > d,\; \K_c$ is still disconnected.
\end{enumerate}

\begin{prop}
\label{discon}
Assume that $f_1(z)= z$ , then for all real numbers $c <-2, \; \K_c$ is a disconnected set.
\end{prop}

\begin{proof}

Since $H_c^{3} (-1,-1)= (-1,-1)$, then $-1 $ belongs to $\K_c$.
On the other hand $z_0= \frac{1+ \sqrt{1-4c}}{2}> 0$ belongs to  $\K_c$ since $(z_0, z_0)$ is a fixed point of $H_c$.

{\bf Claim: The line $i \mathbb{R}$ doesn't intersect $\K_c$.}

Indeed: let $x \in \mathbb{R}$ and $c < -2$, then $f_2 (ix)= -x^2+c < c <-2$ and $\vert f_3(x) \vert = \sqrt{x^2 (-x^2+c)^2 + c^2} > \vert c \vert $,
then by Lemma \ref{twice}, we deduce $ix \not \in \K_c$.
Then we obtain the claim.

\end{proof}

The  connectedness locus of CFS $(f_{n})_{n \in \N}$ is the set $$
\M = \left\{c \in \C~:~ \rm {\K_{c}~is~connected}\right\}.
$$
Another important set is the set
$\M_0$ defined by
$$\M_0 = \left\{c \in \C~:~\right (f_{n}(0))_{n \in \N} {\rm {~~~is~~bounded}}\}.$$

It is clear that under the assumption $f_1(0)=f_1'(0)=0$ the set  $\M$ is a subset of $\M_0.$

\begin{prop}
\label{boundedM}
Assume that  $f_1(0)=f_1'(0)=0$, then
the set $\M_0$ satisfies the following properties
\begin{enumerate}
\item
$\ds D(0,\frac{1}{4}) \subset \M_0 \subset D(0, 2)$.
  \item $\M_0$ is a compact and simply connected set.
\end{enumerate}
\end{prop}

\begin{proof}
By Proposition (\ref{discon1}), we have $\M_0 \subset D(0, 2)$. Now, let $A >1$, then we have the following claim

{\bf Claim}:  $D(0,\frac{A-1}{A^2}) \subset \M_0.$

Indeed, put $r=\frac{A-1}{A^2}$ and  assume that $\vert f_{1}(0) \vert = \vert c \vert <  r < rA$,
then
$\vert f_{2}(0) \vert  \leq r^2+ \vert c \vert < r^2 A^2 +r= rA $.
We deduce easily by induction on $n$ that
$$\vert f_{n}(0) \vert  < rA,\; \forall n \in \mathbb{N}.$$
Hence, we obtain the claim.

Since $\max \{\frac{A-1}{A^2},\; A >1\}= \frac{1}{4}$, we have $D(0,\frac{1}{4}) \subset  \M_0$.

On the other since $\M_0 \subset D(0,2)$, there exists a constant  $R$ (which does not depend on $c$) such that
 if $|f_{k}(0)| > R$, for some integer $k$, then the sequence $(f_n(0))_{n \geq 0}$ is unbounded.
 Putting $f_{n}(0)= \phi_{n}(c)$, we have
 $\M_0= \{c \in \mathbb{C}, \; \vert \phi_{n}(c) \vert \leq R,\; \forall n \in \mathbb{N}\}$.
 Hence  $\M_0$ is a compact and simply connected set.
\end{proof}




\section{Main cardioid: preliminary results}
Here we focus on a subset of the parameter space that generalizes the main cardioid of the Mandelbrot set.
We aim to study the dynamics of CFS with small values of $c$.

For technical reasons we will assume that $f(z)=f_1(z)=z$. Under this restriction the filled-in Julia set of the CFS
corresponds to $\K_c=K^{+}_c \cap \Delta$ where $\Delta$ is the diagonal of $\C^2$ and $K^{+}_c= \{(x, y) \in \C^2,\; H_c^{n}(x, y) \mbox { is  bounded}\}.$ \\
%
%

The goal of the four last sections is the following result:
\begin{thm}
There exists $a>0$ such that for all $0 \leq \vert c \vert <a$, the Julia set $\K_c= K^{+}_c \cap \Delta$ is a quasi-disk, where $\Delta$ is the diagonal $\{ (z,z) \mid z \in \mathbb{C} \}$.
Moreover $c\mapsto \K_c$ is continuous for the Hausdorff topology.
\end{thm}
For a good introduction to quasi-conformal mappings and its applications to dynamics see~\cite{carlgam} and~\cite{ahl} .

{\bf Idea of the proof.}
When $c=0$ the dynamics is simple to understand: there is an invariant torus $\vert x \vert = \vert y \vert =1$ that cuts the diagonal $\Delta$ into a circle.
Through points of that circle $\mathcal{C}$ are local stable manifolds, transverse to $\Delta$. If one can show that these disks are analytic graphs that move holomorphically with $c$, their intersection points with $\Delta$ will define a holomorphic motion of the circle $\mathcal{C}$, hence define a family of quasi-circles $\mathcal{C}_c$. Then it remains to show that the interior of those quasi-circles are in $K^{+}_c$ (which will be a simple consequence of Liouville's theorem) and that the exteriors are in the escaping sets.

{\bf Invariant torus and adapted coordinates.}
In this section we consider the (monomial) map $H_{0}:\binom{x}{y} \mapsto \binom{xy}{x}$ defined on $\C^2$.

\begin{lemm}
When $c=0$,  the torus $\mathbf{T}_0:=\left \{ \binom{x}{y} \in \mathbb{C}^2  ; \vert x \vert = \vert y \vert =1  \right \}$
is a hyperbolic invariant set. Moreover, around each point $\binom{x_0}{y_0} \in \mathbf{T}_0$ there exists an open set with local branches of $y \mapsto y^{\beta_i}$ such that the two maps
\[
\phi_{1}:\binom{x}{y} \rightarrow \frac{x}{y^{\beta_{1}}} \textrm{ and } \phi_{2}:\binom{x}{y} \rightarrow \frac{x}{y^{\beta_{2}}}
\]
are well-defined and semi-conjugate $H_{0}$ to the map
\[
L: \binom{u}{v} \rightarrow \binom{u^{1-\beta_{1}}}{v^{1-\beta_2}},
\]
where  $\beta_{1}:=\frac{1+\sqrt{5}}{2} \approx 1.61$ and $\beta_{2}:=\frac{1-\sqrt{5}}{2} \approx -0.61 $  are the two eigenvalues of the matrix $M= \bigl(\begin{smallmatrix}
1&1\\ 1&0
\end{smallmatrix} \bigr)$.

Using coordinates $(e^{i \alpha}, e^{i \beta})$ on the torus $\mathbf{T}_0$, the restriction of $H_0$ to $\mathbf{T}_0$ coincides with the linear map $M: \binom{\alpha}{\beta} \rightarrow \binom{\alpha+\beta}{\alpha} $.
\end{lemm}

\begin{proof}
The stable and unstable manifolds of the invariant torus are 3-dimensional real analytic manifolds given by the equations $\vert x \vert = \vert y \vert ^{\beta_i}$, where the $\beta_i$ are the roots of $X^2-X-1$. The formulas for the semi-conjugacy come from a simple computation:

\[
\phi_{1} \circ H_{0}\binom{x}{y}= \frac{x.y}{x^{\beta_1}}=x^{1-\beta_1}.y=\left( \frac{x}{y^{\beta_1}} \right)^{1-\beta_1}=\left( \phi_{1}\binom{x}{y}\right)^{1-\beta_1}.
\]
\end{proof}

{\bf Dynamics in the non-perturbated case.}

We will need to visualize four-dimensional neighborhoods of the invariant torus. In order to do this we will draw the relevant domains in the plane $(|x|,|y|)$.

By taking absolute values of the functions $\phi_i$, $i=1,2$, we can easily prove:
\begin{lemm}
\label{SR}
Let $R(x,y):=\frac{\vert x \vert}{\vert y \vert^{\beta_1}}$ and $S(x,y):=\frac{\vert x \vert}{\vert y \vert^{\beta_2}}$.
Then one has
\[
R \circ H_{0}=R^{1-\beta_{1}} \textrm{ and } S \circ H_{0}=S^{1-\beta_{2}}.
\]
\end{lemm}

The dynamics in the plane $(R,S) \in \mathbb{R}^{+} \times \mathbb{R}^{+} $ is given simply by $(R,S) \mapsto (R^{1-\beta_1},S^{1-\beta_2})$. The first figure shows some level curves $(R=cst)$ (looking like graphs $y=\sqrt{x}$) and $(S=cst)$ (looking like hyperbolas $x.y=cst$).



\begin{rem}
\label{K0}

Since $H_0^{(n)}(x, y)= (x^{F_{n}} y^{F_{n-1}}, x^{F_{n-1}} y^{F_{n-2}})$ for all integer $n \geq 2$, we deduce that
$$K^{+}_0=  \{(x, y) \in \C^2; \vert x ^{\beta_{1}} y \vert \leq 1\}= \{(x, y) \in \C^2; S(x, y)\leq 1\} .$$

\end{rem}

We recall that $K^{+}_c$ is the set of points of $\C^2$ with bounded forward orbit for the map $H_{c}$, the escape locus is defined as
$U^+_c=\C^2\setminus K^{+}_c$.
We denote by $\mathbb{B}(0,r)= B_r \times B_r$ the {\bf open bidisk} centered at the origin in $\mathbb{C}^2$, with radius $r>0$.

\begin{defn}[Neighborhood in "good position"]
Let $0<a<1$ , $0<b<1$ and $0<r<1$ and let $H_c$ be a fixed map.
Let us consider the neighborhood of the torus $T_{0}:=\left \{ (x, y) \in \C^2 ; \vert x \vert = \vert y \vert =1  \right \}$ defined by
\[
\mathcal{N}_{a,b}=\{(x,y); 1-a \leq R(x,y) \leq 1+a  \textrm{ and } 1-b \leq S(x,y) \leq 1+b\}.
\]
Then we say that $\mathcal{N}_{a,b}$ is {\bf in good position with respect to the open bidisk} $\mathbb{B}(0,1-r)$ if the following condition is satisfied:
\[
H_{c}(\mathcal{N}_{a,b})=\mathcal{M}_{1} \amalg \mathcal{M}_{2} \amalg \mathcal{M}_{3} \textrm{  (disjoint union) }
\]
where:
\begin{enumerate}
\item $\overline{\mathcal{M}_{1}} \subset \mathbb{B}(0,1-r)$,
\item $\mathcal{M}_{2} = \mathcal{N}_{a,b} \cap H_{c}(\mathcal{N}_{a,b})$,
\item $\overline{\mathcal{M}_{3}} \subset U^{+}_c$ where $ U^{+}_c$ is the set of points with unbounded forward orbit.
\end{enumerate}

\end{defn}

\begin{defn}[Neighborhoods $W^{n}_c$]
Let $W^{0}_{c}:=\mathcal{N}_{a,b}$ be a fixed neighborhood (as above) of the torus $T_0$. Let us define inductively
\[
W^{n+1}_{c} :=W^{n}_{c} \cap H_{c}^{-1}(W^{n}_{c}).
\]

\end{defn}

\begin{lemm}[Case $(c=0)$]
For any $W:=W^{0}_{0}=\mathcal{N}_{a,b}$, with 
$a$ in $[0, 0.4]$ the sequence $(W^{n}_{0})_{n \geq 0}$ is decreasing, all the $W^{n}_{0}$ are homeomorphic to $W$, and the intersection
\[
W^{\infty}_{0}:= \bigcap_{n \geq 0} W^{n}_{0}
\]
 is homeomorphic to a direct product $S^1 \times (\textrm{Annulus}) $.

Moreover, $W^{\infty}_{0}$ coincides with $W \cap W^{s}(T_0)$, where $W^{s}(T_0)$ is the stable manifold of the torus $T_0$ invariant under $T_0$.

\end{lemm}

\begin{figure}[htb]
\begin{center}
\includegraphics[width=6cm]{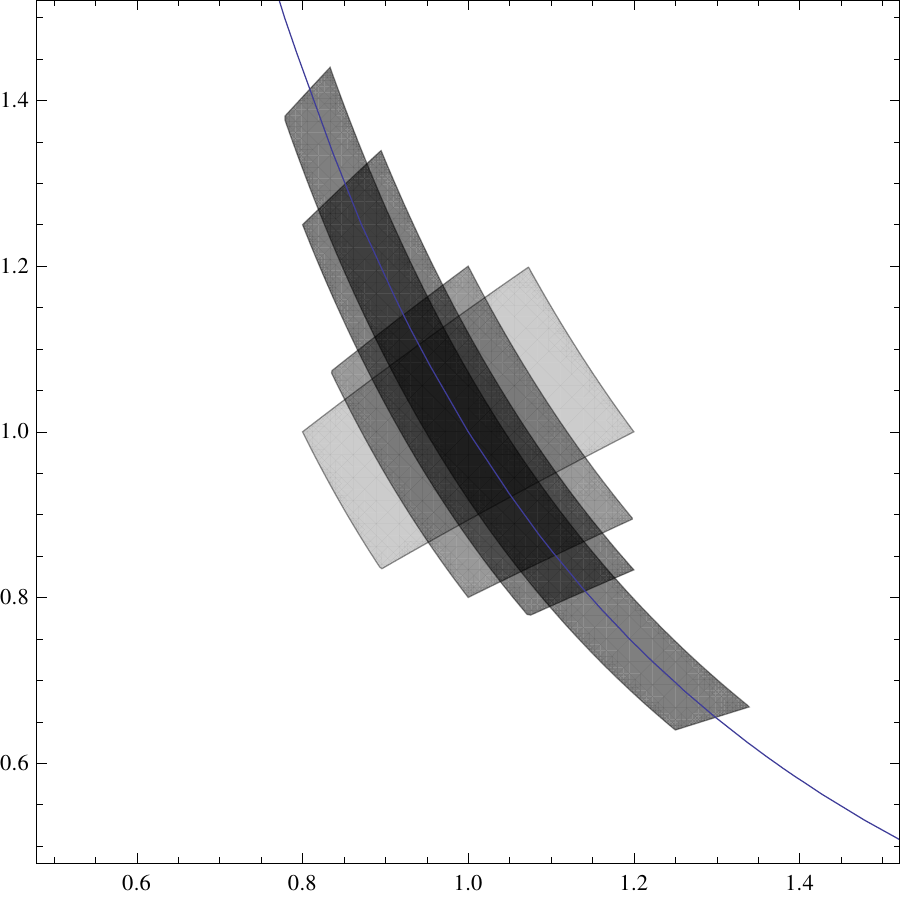}
\caption{Preimages of W and locus S=1}
\label{PreimagesW}
\end{center}
\end{figure}

\begin{proof}
We work in the plane $(R,S)$. Then for any rectangle $\mathcal{R}$ of the form $\mathcal{R}:=\{ 1-a<R<1+a,\textrm{ and } 1-b<S<1+b \}$ (with $0 \leq b \leq 1$ and  $0 \leq a \leq 1$ ) we have $H_{0}^{-1}(\mathcal{R})=\{(1+a)^{-\beta_1} <R<(1- a)^{-\beta_1}, \textrm{ and } (1-b)^{-\beta_2}<S<(1+b)^{-\beta_2} \}$.

Now, we have $(1 - b) < (1 - b)^{-\beta_2}$ and $ (1 + b)^{-\beta_2} < 1 + b$.
 For $a$  small enough (less than $0.4$ is enough), we have $(1 + a)^{-\beta_1} < 1 - a$ and $1 + a < (1 - a)^{-\beta_1}$
Thus $\mathcal{R} \cap H_{0}^{-1}(\mathcal{R})=\{1-a <R<1+a, \textrm{ and }(1 - b)^{-\beta_2} <S<  (1 + b)^{-\beta_2} \}$, which shows that the $W_{0}^n$ are all homeomorphic and form a decreasing sequence converging to $\{ 1-a <R<1+a \} \times \{ 1\}$. Now it suffices to remember that the stable manifold of the torus $\{R=1,S=1\}$ corresponds exactly to $\{S=1\}$.
\end{proof}

{\bf Existence of a neighborhood in good position for $c=0$}

First we show that, by taking $c$ small enough we can find symmetric bidisks around the origin, of radius arbitrarily close to $1$,  that are contained entirely inside $K^+_c$.
\begin{lemm}
\label{ep}
For any $0<\delta<1$ there exists $\epsilon>0$ such that for all $0 \leq \vert c \vert <\epsilon$ the bidisk $\mathbb{B}(0,1-\delta)$ is included in $K^{+}_{c}$.
\end{lemm}
\begin{proof}
Since $\vert xy+c \vert \leq \vert xy \vert + \vert c \vert \leq (1-\delta)^2 + \vert c \vert = (1-\delta)+\delta^2 - \delta + \vert c \vert$, it is enough to take $\epsilon < \delta - \delta^2$.
\end{proof}

For the next Lemma, we keep the same notations as above.
\begin{lemm}
\label{GoodPosition}
Given any neighborhood $\mathcal{N}$ of the torus $T_{0}$, there exists $\delta>0$, $\alpha>0$  and $0<a<1$ , $0<b<1$ such that the following conditions are satisfied:
\begin{enumerate}
\item the bidisk $\mathbb{B}(0,1-\delta)$ is included in $K^{+}_c$ for any $c$ such that $\vert c \vert <\alpha$;
\item $\mathcal{N}_{a,b} \subset \mathcal{N}$;
\item the neighborhood $\mathcal{N}_{a,b}$ is {\bf in good position with respect to} the bidisk $\mathbb{B}(0,1-\delta)$, for all maps $H_c$ with $\vert c \vert <\alpha$.

\end{enumerate}
\end{lemm}

{\bf Beginning of the proof.}

(1) comes from Lemma \ref{ep}.

For the proof of (2) and (3), for any $0<a<1$, we denote by $R_a$ the set
\[
\left \{ (x, y) \in \C^2, \textrm{ such that } 1-a < R(x,y)<1+a \right \},
\]
First we observe the following facts:
\begin{enumerate}
\item for any $a \in (0,1)$  small enough (less than $0.4$ is enough), we have $ 1-a <(1 + a)^{\beta_2} $ and
$ (1 - a)^{\beta_2} <1+a $.
By Lemma \ref{SR} and the fact that $\beta_2= 1- \beta_1$, we have
 $H_0(R_a) \subset int(R_a)$,

\item no bidisk $B_{1-\delta} \times B_{1-\delta}$ intersects a locus of the form
\[
\{ S>1+b \},
\]
\item for any $b \in (0,1)$ and $0<a<b/2$  there exists $\delta_0 >0$ such that for any $0< \delta <\delta_0$ we have:

 the bidisk $B_{1-\delta} \times B_{1-\delta}$ centered at the origin contains the locus
\[ R_a \cap
\{ S<1-b \}.
\]

\end{enumerate}
For this last fact, it suffices to choose $a$ and $b$ such that
$$(\frac{1-b}{ 1-a})^{1/ 1+ \beta_1} < 1- \delta,\;  (1+a) (\frac{1-b}{ 1-a})^{\beta_1/ 1+ \beta_1} < 1- \delta.$$

Now we observe that as $a \rightarrow 0$, $b \rightarrow 0$, the neighborhoods $\mathcal{N}_{a,b}$ converge to the torus $\mathbf{T_0}$ (corresponding to $\vert x \vert = \vert y \vert =1$), thus given any $\mathcal{N}$ one can certainly ensure that $\mathcal{N}_{a_0,b_0} \subset \mathcal{N}$ and notice that $\mathcal{N}_{a,b} \subset \mathcal{N}_{a_0,b_0}$ for all $a,b$ satisfying $a \leq a_0$ and $b \leq b_0$.
Now, once $\mathcal{N}_{a,b}$ is chosen ($a \geq a_0,\; b \geq b_0$), we choose a $\delta$ as above.
Let us call $W:=\mathcal{N}_{a,b}$. Then we see immediately
that $H_0(W) \subset int (R_a)$. Also, the following partition of $R_a$ into three parts will induce on $H_0(W)$ a corresponding partition:
\begin{enumerate}
\item First part: $R_a \cap \{ S<1-b \} \subset \{S \leq 1\} = K^{+}_{0}$,
\item Second part: $R_a \cap \{ 1-b \leq S \leq 1+b \}=W$,
\item Third part: $R_a \cap \{ S>1+b  \} \subset U^+_0$.
\end{enumerate}

On the other hand, we can prove that for $\delta$ small, the set $R_a \cap \{ S > 1+b\}$ is contained in the set $\{(x, y), \vert x \vert > 1+ \delta, \vert  y \vert > 1+ \delta\}$.

For this, it suffices to choose $\delta_0>0$  such that for all  $0 < \delta < \delta_0$

$$ \Bigl(\frac{1+b}{1+a}\Bigr)^{1/(\beta_1- \beta_2)} > 1+ \delta  \textrm{ and }
(1-a ) \Bigl( \frac{1+b}{1+a} \Bigr)^{\beta_1/(\beta_1- \beta_2)} > 1+ \delta.$$

We deduce that

\[
H_0(W) \subset int(R_a) \subset  B(0, 1- \delta)  \cup W  \cup \{(x, y), \vert x \vert > 1+ \delta, \vert  y \vert > 1+ \delta\}.
\]

\begin{lemm}
\label{SubsetOfUPlus}
For any $0<\delta<1$ there exists $\epsilon>0$ such that for all $0 \leq \vert c \vert <\epsilon$ the domain $\{ |x|>1+ \delta \textrm{ and } |y|>1+ \delta\}$ is included in $U^{+}_{c}$.
\end{lemm}
\begin{proof}
Since $\vert xy+c \vert \geq |xy|-|c| \geq (1+\delta)^2-|c| \geq 1+2\delta + (\delta^2 - |c|)$, it is enough to take $\epsilon <  \delta^2$ to ensure that the orbit diverges to $(\infty, \infty)$.
\end{proof}

{\bf End of the proof of Lemma \ref{GoodPosition}}

{\bf Existence of a neighborhood in good position for $c$ small.}

Let $W= \mathcal{N}_{a,b} $ is as defined in \ref{GoodPosition}.
Since for $c=0$ we know that $H_0(\overline{W}) \subset int( \{ 1-a<R<1+a \})$, the same is true for $H_c$ with $\vert c \vert$ small enough. Thus $H_c(W) \subset  int(\{ 1-a<R<1+a \}) \subset \mathbb{B}(0,1-\delta) \cup W \cup \{ |x|>1 + \delta \textrm{ and } |y|>1+ \delta \} $, and this is enough to ensure that $W$ is in good position with respect to
$\mathbb{B}(0,1-\delta)$ for all $H_c$ with $\vert c \vert$ small enough.

This ends the proof of Lemma 5.8.
\hfill $\Box$

As an immediate consequence of the partition of $H_0(W)$ into three parts we have the following useful Proposition:

\begin{prop}[Dynamics in $W$ when $c=0$]
\label{DynamicsPartitionZero}
There is a partition
\[
W=(W \cap int(K^{+}_0)) \amalg \mathcal{W}^S(W) \amalg (W \cap U^{+}_0),
\]
with the following properties:
\begin{enumerate}
\item $W \cap U^{+}= \bigcup_{n \geq 0} U^{n}_{0}(W)$ where $U^{0}_0=U^{+}$ and  $U^{n+1}_0 := H_{0}^{-1}(U^{n}_0(W) \cap W)$, for all $n \geq 0$.
\item $W \cap K^{+}= \bigcup_{n \geq 0} K^{n}_{0}(W)$ where $K^{0}_0=K^{+}$ and $K^{n+1}_0 := H_{0}^{-1}(K^{n}_0(W) \cap W)$, for all $n \geq 0$.
\item $\mathcal{W}^S(W)$ is defined as $\{ p \in (x,y) \in W \mid H^{(n)}_{0}(p) \in W , \forall n \geq 0 \}$.

\end{enumerate}
\end{prop}
\begin{proof}
From Lemma \ref{GoodPosition} we know the fate of points of $W$ under forward iteration:
they eventually fall in $\mathbb{B}(0, 1-\delta) \subset int(K^{+}_0)$ or in $U^{+}_0$, or stay forever in $W$ (and hence, by definition, belong to $\mathcal{W}^S(W)$). From that same Lemma, we know that a point $p \in W$ enters eventually $int(K^{+}_0)$ if and only if one of its forward iterates lands in the connected component of $(H_{0}(W)-W)$ that is inside $\mathbb{B}(0, 1-\delta)$. Similarly  a point $p \in W$ enters eventually $U^{+}_0$ if and only if one of its forward iterates lands in the connected component of $(H_{0}(W)-W)$ that is inside $U^{+}_0$.
\end{proof}

An immediate consequence is:
\begin{prop}
Proposition \ref{DynamicsPartitionZero} is also true for $|c|$ small.
\end{prop}

{\bf The invariant torus $\mathbf{T}_0$ and its stable foliation.}
In this section we set $W:=\mathcal{N}_{a,b}$, where $\mathcal{N}_{a,b}$ is as in Lemma \ref{GoodPosition}.
\begin{lemm}
When $c=0$, the set $\mathcal{W}^S(W)$ defined in \ref{DynamicsPartitionZero} coincides with the set
\[
\{ 1-a< R < 1+a \} \times \{ S=1 \}.
\]
It is also the part of the stable manifold of the invariant torus  $\mathbb{T}_0$ lying in $W$, and is thus a disjoint union of analytic disks, each one being a piece of stable manifold of a point of the torus.
\end{lemm}
\begin{proof}
The first assertions are immediate consequences of the fact that the functions $R$ and $S$ satisfy simple functional equations.
That the stable manifolds are analytic manifolds comes from the standard Stable Manifold theorem for hyperbolic sets (see for example Theorem 6.4.9 in Katok-Hasselblatt \cite{KH}).
\end{proof}

\begin{lemm}
The analytic disks foliating $\mathcal{W}^S(W)$ intersect transversally the diagonal $\Delta:=\{(z,z) \in \mathbb{C}^2, z \in \mathbb{C}\}$ along the unit circle.
\end{lemm}
\begin{proof}
This comes from the fact that $\{ S=1 \} \cap \Delta$ is given by the set $\left \{ (x,x) | \frac{|x|}{|x|^{\beta_2}}=1 \right \}$, which is the unit circle inside the diagonal. Now the local stable manifolds of the torus are locally given by graphs of the form $\frac{x}{y^{\beta_2}}= \textrm{const}$ (where we use local branches of the function $y^{\beta_1}$), and those are transverse to the diagonal.

\end{proof}

\section{Small perturbations}

The main idea of this section is that the invariant torus $\mathbf{T}_0$ will persists under perturbations when $c \neq 0$, and the local stable manifold of the torus will move analytically when $c$ is changed, and will therefore intersect the diagonal along points that move analytically when $c$ changes, thus defining a holomorphic motion of the initial circle $\mathcal{C}:= \mathbf{T}_0 \cap \Delta$.

The general theory of Structural Stability of hyperbolic sets ensures the existence of an invariant torus $\mathbf{T}_c$ for $c$ small enough, such that the maps $H_c$ and $H_0$ are conjugated on the tori.

The relevant Theorem is taken from Katok-Hasselblatt \cite{KH}.
\begin{thm}[Structural stability of Hyperbolic sets]
Let $M$ be a Riemannian manifold and $U \subset M$. Let $\Lambda \subset M$ be a hyperbolic set of the diffeomorphism $f:U \rightarrow M$. Then for any open neighborhood $V \subset U$ of $\Lambda$ and every $\delta >0$ there exists $\epsilon >0$ such that if $f':U \rightarrow M$ and $d_{C 1}(f_{\vert V},f')< \epsilon$, there is a hyperbolic set $\Lambda'=f'(\Lambda') \subset V$ for $f'$ and a homeomorphism $h: \Lambda' \rightarrow \Lambda$ with $d_{C^0}(Id, h)+ d_{C^0}(Id, h^{-1})< \delta$ ($d_{C^i},\; i=0,1$ distances induced by $C^i$ topologies) such that $h \circ f'_{\vert \Lambda'}=f_{\vert \Lambda} \circ h$. Moreover, $h$ is unique when $\delta$ is small enough.

\end{thm}

{\bf Dependence of the local stable manifolds on the parameters.}

We recall that for any $x \in K$ we can define a {\it local stable manifold} $\mathcal{V}_{x}^{-}$ and a {\it local unstable manifold} $\mathcal{V}_{x}^{+}$ (see \cite{Rue} for more details).

\begin{thm}[Persistence]
The manifolds $\mathcal{V}_{x}^{\pm}$ are of class $C^{r}$ when $f$ is of class $C^{r}$. Moreover the stable and unstable manifolds of $\Lambda'$ depend continuously on $f'$ for the $C^r$ topologies.
\end{thm}

We will need a bit more, namely the analytic dependence of the stable manifold $\mathcal{W}^s(W)$:

\begin{lemm}
The local stable manifolds in $W$ depend analytically in $(x,y,c)$.
\end{lemm}

\begin{proof}

One way to prove this Lemma would be to show the analytic dependence of the local stable manifolds of periodic points, and use the density of these periodic points on the invariant torus.
Instead we use here some more general results about persistence of complex laminations due to Berger (see \cite{Ber}). In \cite{Ber}, Theorem 0.3 the author proves the persistence of a complex lamination $\mathcal{L}$ preserved by holomorphic endomorphisms, under the technical condition that the lamination is "0-normally expanded", a condition that is immediately satisfied in our situation where the lamination is made from local leaves of stable manifolds.
Thus the local stable manifolds in $W$ can be described as graphs of analytic functions of the form $f(y,c)$.
\end{proof}

\begin{rem} As observed in the introduction of this section the invariant torus and the hyperbolicity persist. Let us mention that, in addition, topological entropy, exponential mixing and the spectral structure persist under such small perturbations. For more details on the ergodic properties of $H_c$ , we refer the reader to the survey by Guedj \cite{guedj} and the references therein.
\end{rem}

Now we are able to construct the quasi-disk.
\begin{prop}
Let $W$ be defined as in Lemma \ref{GoodPosition}. Then there exists $\epsilon>0$ such that for each $|c|< \epsilon$ we have:
$\mathcal{W}^s(W)$  intersects $\Delta$ along a quasi-disk $D_c$.
\end{prop}

\begin{proof}
The proof is a transversality argument:
through each point $p_0$ on the circle $\mathcal{C}:=\Delta \cap \mathcal{W}^s(W)$ there is a local stable manifold going through it and intersecting transversally the diagonal $\Delta$ in $p_0$.
The invariant torus persists, and thus one can define a continuous function $c \mapsto p_c \in \mathbb{T}_c$ such that $p_0$ is the initial point on the circle $\mathcal{C}$.
Now we know that transverse intersections persist under small perturbations, and also that local stable manifolds are graphs of functions that depend analytically in $c$. Thus the intersection of $W^s(p_c)$ with the diagonal $\Delta$ defines an analytic function $c \mapsto \phi(c,p_0)$ satisfying $\phi(0,p_0)=p_0$, hence a holomorphic motion of the entire circle $\mathcal{C}$.
Indeed, two distinct points $p_0,q_0$ have distinct local stable manifolds passing through them, that stay distinct for $c$ small enough and therefore have distinct intersection points with the diagonal.
\end{proof}

We recall here the definition of a holomorphic motion of a set in $\mathbb{C}$.

\begin{defn}[Holomorphic motion of a set $E \subset \mathbb{C}$]
Let $E$ be a subset of the Riemann sphere having at
least three points. Let $D(0,R)$ be the disk centered at the origin with radius $R$ in $\mathbb{C}$. Then a holomorphic motion of $E$ over $D(0,R)$ is a mapping
\[
f :D(0,R) \times E \rightarrow \hat{\mathbb{C}}
\]
with the following properties:
\begin{enumerate}
\item for a given $\lambda \in D(0,R)$, the map $f_{\lambda}:E \rightarrow \hat{\mathbb{C}}$ given by $f_{\lambda}(z)=f(\lambda,z)$ is injective;
\item the map $f_0$ is the identity;
\item for each $z \in E$, the map $\lambda \mapsto f_{\lambda}(z)$ is holomorphic.
\end{enumerate}
\end{defn}

Now such a holomorphic motion has an automatic extension to the entire Riemann sphere, and more importantly the extension is quasi-conformal. This is the content of the $\lambda-$lemma which we recall here, following the book of de Melo and de Faria \cite{deFaria}:

\begin{thm}[$\lambda$-lemma]

 Let $f :D(0,R) \times E \rightarrow \hat{\mathbb{C}}$ be a holomorphic motion of a set $E \subset \hat{\mathbb{C}}$. Then the map $f_{\lambda}: \overline{E} \rightarrow \hat{\mathbb{C}}$ given by $f_{\lambda}(z)=f(\lambda,z)$ is quasi-conformal for each $\lambda \in D(0,R)$. Moreover, $f$ has an extension to a continuous map $\hat{f}: D(0,R) \times \overline{E} \rightarrow \hat{\mathbb{C}}$ that is a holomorphic motion of the closure of $E$, and such an extension is unique.
\end{thm}

As a consequence, when $c \neq 0$ changes, the initial circle $\mathcal{C}$ is deformed into a topological circle $\mathcal{C}_c$ which by the $\lambda-$lemma is actually a quasi-circle.
The bounded connected component of the complement of the quasi-circle is then a quasi-disk and this concludes the construction.

\section{Proof of Main result}
The goal of this section is to prove that $K^+_c \cap \Delta$ coincides with the quasi-disk $D_c=\mathcal{W}^s(W)\cap\Delta$ defined above, for $c$ small enough.

In order to do this we will need several preliminary results, that will allow us to compare the dynamics in the case $c=0$ to the dynamics when $c \neq 0$ is small enough.

One important tool that we will use in the proof is the theory of the {\it crossed mappings}, first developed by Hubbard and Oberste-Vorth in \cite{HOVII}, and that we will now review.

{\bf Crossed mappings.}

In the domain $W$, small perturbations of $H_0$ have expanding and contracting directions that allow the use of {\it graph transform} methods. A general framework for such methods is given by the theory of {\it crossed mappings} as described in \cite{HOVII}.

Here we recall the main results of \cite{HOVII} concerning such mappings.

 Let $B_1 = U_1 \times V_1$ and $B_2 = U_2 \times
V_2$ be bidisks.

\begin{defn}
A crossed mapping from $B_1$ to $B_2$ is a triple $(W_1, W_2, f)$, where
\begin{enumerate}
\item
$W_1 \subset U'_1 \times V_1$ where $U'_1 \subset U_1$ is a relatively compact open subset,
\item
$W_2 \subset U_2 \times V'_2$ where $V'_2 \subset V_2$ is a relatively compact open subset,
\item
$f: W_1 \to W_2$ is a holomorphic isomorphism, such that for all $y \in V_1$, the mapping
\[
pr_1 \circ f\vert_{W_1\cap (U_1 \times \{y\})}: W_1 \cap (U_1 \times \{y\}) \rightarrow U_2
\]
is proper, and the mapping
$$
pr_2 \circ f^{-1}\vert_{W_2\cap(\{x\}\times V_2)}: W_2 \cap (\{x\}\times V_2) \rightarrow V_1
$$
is proper.
\end{enumerate}
\end{defn}

{\bf Cone fields in the tangent space.}

 For any bidisk $B= U
\times V$, consider the horizontal cone field $C_{(x,y)} \subset T_{(x,y)}B$ defined by
\[
C_{(x,y)} = \left \{ {(\xi,\eta) \in T_{(x,y)}B} \textrm{ such that }\vert(x,\xi)\vert_U \geq \vert(y,\eta)\vert_V \right \},
\]

where  $T_{(x,y)}B$ denotes the tangent space at the point $(x,y)$.
One can also define a {\it vertical cone field} by reversing the inequality.

As in \cite{HOVII} we call smooth curves and surfaces in a bidisk $B = U \times V$ {\it  horizontal-like}
of {\it vertical-like} if their tangent spaces are in the horizontal or vertical cone respectively at
each of their points.

The next Proposition essentially says the following: if one has an infinite sequence of crossed-mappings, where the vertical direction is contracted and the horizontal is expanded, then the set of points inside one bidisk, say $B_0$, whose entire forward orbit always stays within the $B_i$ will be an almost vertical analytic graph that we should understand as a piece of stable manifold. A similar situation will occur for unstable manifolds.
The next Proposition is a formal description of this argument.

\begin{prop}
Let
$$
\dots B_{-1} = U_{-1} \times V_{-1}, B_0 = U_0 \times V_0, B_1 = U_1 \times V_1, \dots
$$
be a bi-infinite sequence of bidisks, and $f_i: B_i \rightarrow B_{i+1}$ be crossed mappings of degree 1,
with  $U'_i$ of uniformly bounded size in $U_i$, where the $U_i, U'_i, V_i, V'_i$ are as in the definition of crossed mappings. Then for all $m \in \mathbb{Z}$,
\newline
(1) the set

\begin{align*}
W^S_m = \{ (x_m,y_m) \mid & \textrm{ there exist }(x_n,y_n) \in B_n\\
&\textrm{for all }n \geq m \textrm{ such that }f_n(x_n,y_n) = (x_{n+1},y_{n+1})\}
\end{align*}

is a closed vertical-like Riemann surface in $B_m$, and $pr_2: W^S_m \rightarrow V_m$ is an isomorphism;
\newline
(2) the set

\begin{align*}
W^U_m = \{(x_m,y_m) \mid &\textrm{there exist $(x_n,y_n) \in B_n$}\cr
&\textrm{for all $n < m$ such that $f_n(x_n,y_n) = (x_{n+1},y_{n+1})$}\}
\end{align*}
is a closed horizontal-like Riemann surface in $B_m$, and $pr_1: W^S_m \rightarrow U_m$ is an isomorphism.
\newline
(3) Moreover, the sequence
\[
(x_m,y_m) := W^S_m \cap W^U_m\,,\quad m\in\mathbb{Z},
\]
is the unique bi-infinite sequence with $(x_m,y_m) \in B_m$ for all $m \in \mathbb{Z}$, and $f_m(x_m,y_m)
\allowbreak = (x_{m+1},y_{m+1})$.

\end{prop}

{\bf Application to the map $H_0$ and $H_c$ for $c$ small.}

\begin{lemm}
\label{HzeroCrossedMapping}
Let $z \in \mathbf{T}_0$. Then one can choose a bidisk $B_i, i \in \mathbb{Z}$ around each point $z_i, i \in \mathbb{Z}$ of the full orbit of $z$, such that $H_{0}: B_i \rightarrow B_{i+1}$ is a crossed mapping of degree one. Moreover for any small enough $c$, the map $H_c: B_i \rightarrow B_{i+1}$ is also a crossed mapping of degree one.
\end{lemm}

\begin{proof}
Locally around each $z_i$, we can find a local branches of the functions $y^{\beta_1}, y^{\beta_2}$, and therefore define functions $u,v$ such that in a bidisk $\vert u \vert \leq \alpha, \vert v \vert \leq \beta$ the map $H_0$ is given by
\[
\binom{u}{v} \mapsto \binom{u^{1-\beta_1}}{v^{1-\beta_2}},
\]
where one coordinate is expanded and the other is contracted, so the map is a crossed mapping that sends the "vertical boundary" $\vert u \vert = \alpha$ strictly outside of the bidisk.
By compactness of the torus $\mathbf{T}_0$ we can ensure that the images of all vertical boundaries of the $B_i$ are away from the initial bidisk by a fixed strictly positive distance. Therefore, by continuity, for $c$ small enough, the restriction of $H_c$ to the same family of bidisks will still be a crossed mapping of degree one.
\end{proof}

{\bf Main Result.}
We recall the main result of the section.
\begin{thm}
The intersection $K_c^+ \cap \Delta$ coincides with the closed quasidisk $\overline{\mathcal{D}_c}$, for $c$ small enough.
\end{thm}
{\bf Proof of the Theorem.}

{\bf Easy inclusion: $\mathcal{D}_c \subset (K^+_c \cap \Delta_c)$.}
This is simply a matter of observing that the boundary of $\mathcal{D}_c$ is included in $K^+_c$ since it is contained in the set $\mathcal{W}^s_c(W)$ of points with forward orbit entirely contained in $W$.
Then, an immediate application of the maximum principle yields that the points of the interior of the quasi-disk have also bounded forward orbits.

{\bf Second inclusion: $K^+_c \cap \Delta \subset \mathcal{D}_c.$}
We will prove now by contradiction that the exterior of the quasi-disk does not intersect $K^{+}_c$, and this alone will require several lemmas.

Recall that for all $c$, the stable manifold $W^s(T_c)$ is the set of all points in $\C^2$ whose forward orbit converges towards $T_c$.
\begin{lemm}
We have $W^{s}_c(W)=W^{s}(T_c)$.
\end{lemm}

{\bf Proof:}
Let $p \in W^{s}_c(W)$. Take the forward orbit $\mathcal{O}:=(p, H_c(p), H_c^{(2)}(p), \ldots)$. Then we observe that $\mathcal{O}$ is an $\varepsilon$-pseudo orbit for $H_0$,
where $\varepsilon$ depends only on $|c|$.
By the Shadowing Lemma (see for example Theorem 18.1.3 in Katok-Hasselblatt \cite{KH}), there exists a true $H_0-$orbit of a point $p_0$ which is $\varepsilon$-close to $\mathcal{O}$. Moreover $p_0 \in W^s_0(W)=W^s(T_0)$. Hence $p_0 \in W^s(\tilde{p_0})$ for some point $\tilde{p_0} \in T_0$.

Now by Lemma \ref{HzeroCrossedMapping} we can take a bi-infinite sequence $B_i, i \in \mathbb{Z}$ of bidisks such that both $H_0 :B_i \rightarrow B_{i+1}$ are crossed mappings of degree 1. Then we know that $p$ has a forward orbit that stays in all the $B_i$ for $i \geq 0$, hence it belongs to $W^s(B_0)$ (which is the set of points of $B_0$ that has a forward orbit staying in all the $B_i$ for $i \geq 0$). But $W^s(B_0)$ also contains the point $q \in W^s(B_0) \cap W^u(B_0)$ which by definition is a point of $\cap_{n \in \mathbb{Z}} H^{n}_c(W)$.
Now we use a lemma which is a direct consequence of Theorem 7.4 in \cite{Robinson} "Stability of a hyperbolic invariant set":
\begin{lemm}
$T_c= \bigcap_{n \in \mathbb{Z}} H^{(n)}_c(W)$
\end{lemm}
Thus we deduce that $q \in T_c$ and that finally $q \in W^s(q)$.

\begin{lemm}
$W^s(T_c) \cap \Delta$ is equal to the boundary of the quasi-disk.
\end{lemm}
{\bf Proof.}
In $W$, the stable manifolds belonging to $W^s(T_0)$ are transverse to the diagonal $\Delta$ and are pieces of graphs of functions analytic in $c$. Small perturbations of such graphs stay transverse to $\Delta$, so any point $m \in \Delta \cap W^s(T_c)$ is the continuation of a point $m_0 \in \Delta \cap W^s(T_0)=\mathcal{C}_0$ where $\mathcal{C}_0$ is the boundary of the disc that is $\Delta \cap K^{+}_0$. Thus $\Delta \cap W^s(W)$ coincides with the boundary of the quasi-disk (which had been previously constructed as the holomorphic motion of $\mathcal{C}_0$).

{\bf End of Proof.}
Now it remains to show that the exterior of the quasi-disk does not contain points of $int(K_c^+)$. Let us consider the annulus $\mathcal{A}$ defined as the complement in $W$ of the quasi-disk. Since we know that $W$ is the disjoint union
$$(W \cap int(K_c^+)) \amalg (W^s(W)) \amalg (W \cap U_c^+),
$$
the open annulus $\mathcal{A}$ is the disjoint union  of two open sets $\mathcal{A} \cap int(K_c^+)$ and $\mathcal{A} \cap U_c^+$. By connectedness, $\mathcal{A}$ can only consist of points in $int(K_c^+)$ alone, or $U^+$ alone. Since very far from $W$ in $\Delta$ we have only points in $U_c^+$ we deduce at once that $\mathcal{A}$ contains no point in $int(K_c^+)$, which concludes the proof.

In summary we showed that $\Delta \cap K_c^{+}$ is exactly a quasi-disk, for $c$ small enough.

\qed

\begin{rem}
With minor changes in the proof above, one could also show that for $\vert c \vert$ small enough, the intersection $\K_c \cap \mathcal{P}$ is a quasi-disk, where $\mathcal{P}$ is the parabola $\{(z^2,z); z \in \C \}$.

\end{rem}

{\bf Generalization.}

All results can be  extended  easily to the sets associated of the dynamics of $h(x, y)= (x^a y^b+ c, x)$ where $a, b$ are positive real numbers.

\newpage
 \section*{Acknowledgement}

The first author would like to express thanks to the colleagues of the departamento de Matem\'atica IBILCE-UNESP and departamento de matem\'atica of COMPINAS for the warm hospitality during his visit.
He was supported by Cooperation Brasil-France, and also by Capes-COFECUB 661/10.

The second author would like to thank the IME-USP (S\~ao Paulo)  for supporting his research through the project Fapesp (Num. 2011/12 650-4).

O. Sester was supported by Capes-COFECUB 661/10 and also Fapesp Project 2007/06896-5 and he would like to express his sincere thanks to the Mathematics department of the IBILCE-UNESP for the hospitality during his visit.

A. Messaoudi  would like to express thanks to Pascal Hubert and John Hubbard for fruitful discussions .
He was supported by Capes-COFECUB 661/10,  by Brazilian CNPq grant
305939/2009-2 and also by
 Fapesp project:  2011/23199-1 .

\newpage

\bibliographystyle{apalike}
\setlength{\parsep}{0cm} \small

\end{document}